\documentclass[12pt,a4paper]{article}
\usepackage[OT1]{fontenc}
\usepackage[latin1]{inputenc}
\usepackage[english]{babel}
\usepackage{amsfonts,amsmath,relsize,amsthm,units,amssymb,bbm,newtxtext,newtxmath}
\usepackage[nottoc]{tocbibind}
\usepackage[footskip=.5in]{geometry}
\usepackage{xcolor,enumitem}
\usepackage[colorlinks]{hyperref}

\newtheorem*{theorem*}{Theorem}
\newtheorem{teo}{Theorem}[section]
\theoremstyle{definition}
\newtheorem{prop}[teo]{Proposition}
\newtheorem{lem}[teo]{Lemma}
\newtheorem{coro}[teo]{Corollary}
\newtheorem{defi}[teo]{Definition}
\newtheorem{rem}[teo]{Remark}
\newtheorem{ejem}[teo]{Example}
\newtheorem{problem}[teo]{Problem}

\DeclareMathOperator\ad{ad}
\DeclareMathOperator\Ad{Ad}

\DeclareMathOperator\re{Re}

\DeclareMathOperator\LU{Length}

\DeclareMathOperator\di{dist}

\DeclareMathOperator\tr{Tr}

\DeclareMathOperator\su{\mathfrak{su}_n(\mathbb C)}

\DeclareMathOperator\sop{\mathrm{supp}}
\DeclareMathOperator\g{\mathfrak{g}}
\DeclareMathOperator\s{\mathfrak{k}}
\DeclareMathOperator\ah{\mathfrak{h}}

\DeclareMathOperator\rad{\mathrm{Rad}(\g)}
\DeclareMathOperator\B{\mathfrak{B}}

\newcommand{\todoi}[1]
  {\vspace{5 mm}\par \noindent \marginpar{\begin{center}\LARGE\hspace*{-4cm}I\end{center}} \framebox{\begin
  {minipage}[c]{0.8 \textwidth} \tt\color[rgb]{0,0.5,0} #1
\end{minipage}}\vspace{5 mm}\par}
\newcommand{\todog}[1]
  {\vspace{5 mm}\par \noindent \marginpar{\begin{center}\LARGE\hspace*{-4cm}G\end{center}} \framebox{\begin
  {minipage}[c]{0.8 \textwidth} \tt\color[rgb]{0,0,0.5} #1
\end{minipage}}\vspace{5 mm}\par}

\begin{document}

\title{\vspace*{-2cm}Lie groups with a bi-invariant distance\footnote{2020 MSC Primary 22E65, 58B20; Secondary 53C22}.}
\date{}
\author{Gabriel Larotonda and Ivan Rey\footnote{Supported by CONICET and ANPCyT, Argentina},\footnote{Instituto Argentino de Matem\'atica (CONICET) and Facultad de Ciencias Exactas y Naturales, Universidad de Buenos Aires. e-mails: glaroton@dm.uba.ar, ivanrey1988@gmail.com}}

\maketitle

\abstract{\footnotesize{\noindent We prove that a Lie group $G$ admitting a bi-invariant distance is necessarily of the form $G=H\times K$, where $H$ is a connected abelian Lie group and $K$ is a connected compact Lie group with discrete centre. Moreover, the distance arises from a unique $\Ad$-invariant Finsler norm on the Lie algebra $\mathfrak g=Lie(G)$ as the infimum of lengths of rectifiable paths. The distance-minimising paths are left or right translates of one-parameter subgroups (though not uniquely so when the norm fails to be smooth or strictly convex). Following Milnor, we define a notion of sectional curvature $\sec(\pi)$ for any $2$-plane $\pi\subset\g$; this curvature is non-negative and vanishes precisely on abelian planes. When the norm is strictly convex, $\sec(\pi)=0$ if and only if $\pi$ is abelian. We obtain finer characterisations of vanishing curvature in the non-strictly-convex case.\footnote{{\bf Keywords and phrases:} bi-invariant metric, curvature, Finsler metric, geodesic, Lie group, one-parameter group, skew-Hermitian matrix, unitary group, unitarily invariant norm}}}

\setlength{\parindent}{0cm} 

\tableofcontents
\section{Introduction}

The study of Lie groups equipped with invariant metrics has a rich history, dating back to foundational works in Riemannian geometry and continuing through extensions to more general metric structures. In the Riemannian setting, Milnor's seminal 1976 paper~\cite{milcur} characterized connected Lie groups admitting bi-invariant Riemannian metrics as direct products of compact semisimple Lie groups and abelian vector groups, with non-negative sectional curvature and geodesics realized as one-parameter subgroups. This result built on earlier observations that compact Lie groups always admit bi-invariant metrics (via averaging), while non-compact ones do so only under restrictive conditions. Milnor's curvature formulas, involving the Lie bracket, highlighted the interplay between algebraic structure and geometric properties, such as the covariant derivative being half the adjoint representation.

Subsequent research extended these ideas beyond the Riemannian case. For instance, semi-Riemannian (including Lorentzian) metrics on Lie groups were explored, generalizing Milnor's curvature computations to indefinite signatures, with applications to relativity and pseudo-Riemannian symmetric spaces (see for instance the book by O'Neill \cite{onliegroupssemi}). In the context of non-negative curvature, Tapp~\cite{tapp04} classified left-invariant metrics on low-dimensional compact Lie groups like SO(3) and U(2), showing that many arise via Cheeger's deformation technique---shrinking along subgroup chains---while others, such as certain "twisted" metrics on U(2), do not. These classifications reveal a broader landscape of metrics with non-negative sectional curvature, often strictly containing the bi-invariant ones.

More recently, attention has turned to (smooth) Finsler metrics, which generalize Riemannian structures by allowing non-quadratic norms on tangent spaces. Latifi and Toomanian  \cite{lafiti} studied bi-invariant Finsler metrics on Lie groups, proving that absolutely homogeneous ones render the group a symmetric Finsler space of Berwald type, with geodesics coinciding with those of an underlying bi-invariant Riemannian metric. They derived an explicit flag curvature formula, extending Milnor's sectional curvature to the Finsler setting, and showed that such metrics exist on non-simple or higher-rank compact Lie groups but reduce to Riemannian ones on simple rank-one groups. This aligns with Berestovskii's results~\cite{beres} on homogeneous manifolds with intrinsic metrics, which imply that bi-invariant distances on locally contractible topological groups arise from Ad-invariant Finsler norms on their Lie algebras.

In our prior work~\cite{larey} we established properties of Ad-invariant Finsler norms on skew-Hermitian matrices, including geometry of spheres and subdifferentials. Hofer's metric, a Finsler-type distance arising from symplectic geometry, has been studied on compact Lie groups~\cite{larmi}, revealing minimal geodesics and distance bounds that turned out to be helpful insights in this systematization project.

This paper builds on these foundations, studying the geometry of Lie groups with bi-invariant distances in the general Finsler context, without assuming smoothness or full homogeneity of the norm. We consider an Ad-invariant Finsler norm on the Lie algebra \(\g = Lie(G)\), a subadditive, positively homogeneous map \(|\cdot| : \g \to [0, \infty)\) that is non-degenerate (\(|x| = 0 \Leftrightarrow x = 0\)). Such a norm induces the vector-space topology on \(\g\) and, via left/right translation, a bi-invariant (possibly non-reversible) distance on \(G\) given by the infimum of lengths of rectifiable paths. Conversely, any bi-invariant distance on a connected locally compact locally contractible topological group arises in this way from a unique Ad-invariant Finsler norm (Theorem~A below).

Our aims are twofold: to characterize Lie groups admitting bi-invariant distances and to develop a notion of curvature for such geometries, particularly characterizing flat 2-planes. We adopt the metric geometry viewpoint, where curvature measures distance distortion relative to model spaces. Adapting Milnor's ideas~\cite[p.~101]{milnorm}, we define curvature along pairs of directions in \(\g\) (see also~\cite{cocoeste} and \cite{larro} for a hyperbolic geometry analogue). In terms of complexifications, if \(G\) is a compact complex Lie group with Lie algebra \(\g\) and \(\mathfrak{k}\) its real compact form (\(\g = \mathfrak{k} \oplus i\mathfrak{k}\)), our work examines the geometry of \(G_\mathbb{R} \simeq \mathfrak{k}\) with a (non-Riemannian) bi-invariant distance, complementing studies of homogeneous spaces like \(G / G_\mathbb{R} \simeq i\mathfrak{k}\).

\medskip

What follows is an overview of the organization and key results. Motivated by Milnor's Riemannian characterization, we address whether a similar decomposition holds for general bi-invariant distances. Affirmatively, the main result of Section~\ref{groups} is:

\begin{theorem*}[A] If \((G, d)\) is a connected, locally compact, locally contractible topological group with a bi-invariant intrinsic distance, then \(G = H \times K\) is the product of a connected abelian Lie group \(H\) and a connected compact Lie group \(K\) with discrete center. In particular, \(Lie(G) = Z(\g) \oplus \mathfrak{k}\) where \(\mathfrak{k}\) is a compact Lie algebra.
\end{theorem*}

Locally, the distance between \(e^{sx}\) and \(e^{sy}\) (\(x, y \in \g\), \(s \in \mathbb{R}\)) equals the norm of the Baker--Campbell--Hausdorff series of \(e^{sy} e^{-sx}\),
\[
d(e^{sx}, e^{sy}) = | BCH(sy, -sx) | = | sy - sx - \frac{s}{2} [y, x] + \dots |.
\]
This was established in prior work~\cite{lar19}, refined here using distance invariance (Equation~\eqref{dibch}). If \([x, y] = 0\), then \(d(e^{sx}, e^{sy}) = |sy - sx|\) for small \(s\). To explore this relation deeply, Section~\ref{sec:s} adapts Milnor's approach~\cite[p.~101]{milnorm}, defining sectional curvature for a 2-plane \(\pi \subset \g\) as
\[
S(x, y) = 6 |y - x|^2 \lim_{r \to 0^+} \frac{r |y - x| - d(e^{rx}, e^{ry})}{r^2 d(e^{rx}, e^{ry})}.
\]
Using BCH expansions, we express \(S(x, y)\) in metric terms and prove:

\begin{theorem*}[B] For \(x, y \in \g\) and any Ad-invariant norm on \(\g\),
\[
S(x, y) = 6 |y - x| \lim_{r \to 0^+} \frac{d(e^{r^2 x}, e^{r^2 y}) - r d(e^{rx}, e^{ry})}{r^4} = -\frac{|y - x|}{4} \max_{\varphi \in N_{y-x}} \varphi([x, [x, y - x]]) \ge 0,
\]
where \(N_{y-x}\) is the set of unit-norm functionals with \(\varphi(y - x) = |y - x|\).
\end{theorem*}

Milnor observed that for Riemannian metrics with normalized \(x, y\), \(S(x, x + y) = \langle R(x, y) y, x \rangle\), where \(R\) is the curvature tensor. Theorem~B characterizes this via norming functionals (subdifferential of the norm), generalizing to Finsler cases.

In the final section, we analyze flatness using root systems in \(\g\)'s compact factor, proving in Theorem~\ref{fs}:

\begin{theorem*}[C] Let \(x, y \in \g\). Consider:
\begin{enumerate}
\item[(1)] \(\varphi([x, [x, y]]) = 0\) for any \(\varphi\) norming \(y - x\).
\item[(2)] \(d(e^{sx}, e^{sy}) = |sy - sx|\) for sufficiently small \(s\).
\item[(3)] \(S(x, y) = 0\).
\end{enumerate}
Then (1) \(\Leftrightarrow\) (2) \(\Rightarrow\) (3). Moreover:
\begin{enumerate}
\item[a)] If the norm is smooth, all conditions are equivalent.
\item[b)] If the norm is strictly convex, all are equivalent to \([x, y] = 0\).
\end{enumerate}
\end{theorem*}

We provide an example showing sharpness, drawn from orthogonal/unitary groups with spectral or trace norms (wich are neither smooth nor strictly convex). The distance condition is related to a certain path contained in a face of the sphere of the Finsler norm, as explained in Section \ref{sec:s}.  We finish the paper with a convenient definition of sectional curvature $sec(\pi)$, $\pi\subset\g$, that takes into account the non-symmetry of the spheres (and extends the Riemannian defintion). We show that under reasonable normalization hypothesis, in this setting one always has $0\le sec(\pi)\le 1$.

\medskip

The geometry of the homogeneous spaces of groups with bi-invariant distances is part of our forthcoming research, based on these results.

\section{Groups with bi-invariant distances}\label{groups}

In this section we describe connected, locally compact, locally contractible  groups that admit a particular distance; by the theorem of Montgomery and Zippin about Hilbert's fifth problem (see \cite{montgomery}), these are in fact finite dimensional Lie groups.

\begin{defi}[Bi-invariant intrinsic distances] A (non-necessarily reversible) distance is a function $\di:X\times X\to \mathbb R_{\geq 0}$ with all the properties of a distance map, except that for some $x,y$ it might be $\di(x,y)\ne \di(y,x)$. As such, it defines a topology on $X$ in the usual fashion. A (possible non-reversible) distance in a topological space is \emph{intrinsic} if its topology is equal to the original topology of the given space. More details on asymmetric distances can be found in \cite{mennu2} and the references therein. If $\di$ is a metric in a topological group $G$, we say that it is \emph{bi-invariant} if
\begin{equation}\label{bi-invd}
\di(gh,gk)=\di(h,k)=\di(hg,kg) \quad \forall \, g,h,k\in G.
\end{equation}
\end{defi}
We will first study the smooth case, so in this section we discuss $\Ad$-invariant norms in Lie groups $G$, and obtain a characterization in terms of its Lie algebra. Denote with $TG$ the tangent bundle of a Lie group $G$, which can be identified via left or right multiplication with $G\times \g$. Denote also with $L_g$ the differential of $\ell_g$, left multiplication by $g\in G$ in the group $G$.

\begin{rem}[Left-invariant metrics]\label{norrma}
For a given Finsler norm in  $\g$ and $g\in G$  the \emph{left-invariant Finsler metric} is defined as $|v|_g=|L_g^{-1}v|$ for $v\in T_gG$, with $|\cdot|_*:TG\to \mathbb R_{\ge 0}$. If $g,h\in G$ then
\[
|L_hv|_{hg}=|L_{hg}^{-1}L_h v|=|L_g^{-1}v|=|v|_g\quad \textrm{ for }v\in T_gG.
\]
The map $(g,v)\mapsto |v|_g=|g^{-1}v|$ is continuous as a map from $TG$ to $\mathbb R$. Any left-invariant Finsler metric in $TG$ can be obtained with this procedure. If the Finsler norm is $\Ad$-invariant then
$$
|R_gv|_{gh}=|L_{gh}^{-1}R_h v|=|\Ad_h^{-1}L_{g^{-1}}v|=|L_g^{-1}v|=|v|_g.
$$
and we say that the Finsler metric in $TG$ is \emph{bi-invariant}.
\end{rem}

\begin{defi}[Rectifiable paths and length]\label{recti}
A curve $\alpha:[a,b]\to G$ is \textit{rectifiable} if $\alpha$ is differentiable a.e. in some chart of $G$ and $t\mapsto |\alpha'(t)|_{\alpha(t)}$ is Lebesgue integrable. For piecewise smooth or rectifiable arcs $\alpha:[a,b]\to G$, we define the \textit{length} of $\alpha$ as
$$
\LU(\alpha)=\int_a^b |\alpha'(t)|_{\alpha(t)} dt=\int_a^b |L_{\alpha}^{-1}{\alpha}'|.
$$
\end{defi}

\begin{defi}
For $g,h\in G$, consider the infima of the lengths of such arcs joining $g,h$ in $G$,
$$
\di(g,h)=\inf\{ \LU(\alpha)\, |\, \alpha:[a,b]\to G\textrm{ is rectifiable }, \alpha(a)=g,\alpha(b)=h\}.
$$
It is straightforward to check that $\di:G\times G\to \mathbb R_{\ge 0}$ is a (non-necessarily reversible) distance: i.e. it is finite, obeys the triangle inequality and it is non-degenerate. However it might be that $\di(g,h)\ne \di(h,g)$ for some $h,g\in G$.
\end{defi}

\begin{rem} Since the metric in $TG$ is left-invariant, then $L(g\alpha)=L(\alpha)$ for any $g\in G$, and thus the distance is left-invariant: only the first equality in \eqref{bi-invd} holds. Moreover, by the smoothness of the map $L$ and the local compacity of the group $G$, the topology induced by this distance is equivalent to the original topology of the Lie group. 

This distance is reversible if and only if $|\cdot|$ is fully homogeneous, for in this case
\[
L(\alpha h)=\int_a^b|L_{\alpha h}^{-1}R_h \alpha'|=\int_a^b |\Ad_{h^{-1}} L_{\alpha}^{-1}\alpha'|=\int_a^b |L_{\alpha}^{-1}\alpha'|=L(\alpha) \quad \textrm{ for any } h \in G.
\]
\end{rem}

\begin{rem}If the Finsler norm is $\Ad$-invariant, the length of paths is bi-invariant, and then the distance in $G$ is bi-invariant as in equation \eqref{bi-invd}. The property $\di(g,h)=\di(g^{-1},h^{-1})$ can be easily established \textit{only} in the case of fully homogeneous norms, by means of the equation
$(\gamma^{-1})'=-\gamma^{-1}\gamma'\gamma^{-1}$. 
\end{rem}

\subsection{Lie groups with $\Ad$-invariant Finsler norms}

We first recall an abtract analogue of a result obtained in another setting:

\begin{lem}\label{desilie}
Let $G$ be a connected Lie group with Lie algebra $\mathfrak g$. Assume that $\mathfrak g$ admits an $\Ad$-invariant Finsler  norm $|\cdot |$. Let $0\ne v \in G$  and $\varphi\in\mathfrak g^*$ be of unit norm with $\varphi(v)=|v|$. Then for any $x\in\mathfrak g$ we have  $\varphi ([v,x])=0$ and $\varphi([x,[x,v]])\le 0$.
\end{lem}
\begin{proof}
The proof is quite simiilar to that given in \cite{larey} in the context of skew-Hermitian matrices, we include it here since we will modify it a bit later on. For any $s\in\mathbb R$ we have
\[
|v|=|\Ad_{e^{s\,  \ad x}}v|\ge \varphi( \Ad_{e^{s\,  \ad x}}v)=\varphi(v)+s \varphi([x,v])+\frac{1}{2}s^2 \varphi([x,[x,v]])+O(s^3).
\]
Since $\varphi(v)=|v|$, dividing by $s>0$ and letting $s\to 0^+$ we have $\varphi([x,v])\le 0$, dividing by $s>0$ and letting $s\to 0^-$ we get $\varphi([x,v])\ge 0$. This proves the first claim. Then
\[
\frac{1}{2}s^2 \varphi([x,[x,v]])+O(s^3)\le 0
\]
and dividing by $s^2$ and letting $s\to 0$ we get the second claim.
\end{proof}

\begin{defi}[Killing form, nilpotents and ideals]
The \emph{Killing form} $\B:\g\times \g\to\mathbb R$ is the bilinear form defined as
\[
\B(x,y)=\tr(\ad x\circ \ad y),\qquad x,y\in\mathfrak g.
\]
Due to the Jacobi identity $\ad[x,y]=\ad x\circ \ad y- \ad y\circ \ad x$ and the ciclicity of the trace map we have
\begin{equation}\label{killcyc}
\B([x,y],z)=\tr(\ad[x,y]\circ \ad z)=\tr(\ad x \circ \ad [y,z])=\B(x,[y,z]).
\end{equation}
The \emph{Killing ideal} is defined as
\[
\mathfrak g^{\perp}=\{x\in \mathfrak g: \B(x,y)=0 \quad \forall y\in \mathfrak g\},
\]
it is an ideal due to the identity \eqref{killcyc}. The \emph{center} $Z(\g)\subset \mathfrak g$ is defined as those $x$ such that $\ad x=0$. The \emph{radical} $\rad$ of $\g$ is its maximal solvable ideal. We have
\begin{equation}\label{radic}
Z(\g)\subset \g^{\perp}\subset \rad.
\end{equation}
The first inclusion is clear, the second one comes from Cartan's criterion of solvability \cite[page 20]{lieH}: an ideal $\mathfrak a$ is solvable if and only if $\B(v,z)=0$ for all $v\in [\mathfrak a,\mathfrak a]$ and all $z\in \mathfrak a$.
\end{defi}

\begin{rem}\label{norposta}
If $|\cdot|$ is a Finsler norm on a real vector space, then $|x|'=|x|+|-x|$ is an actual norm. If the original Finsler norm is $\Ad_G$ invariant, then this norm is also $\Ad_G$ invariant. Hence for the purposes of what's left in this subsection, we can assume when needed that we are dealing with a norm.
\end{rem}

\subsubsection{Hermitian operators, numerical range and spectrum}

\begin{rem}[Complexification, Taylor norms and spectrum]\label{taylor} Let $x,y\in \mathfrak g$, and for a given $\Ad$-invariant Finsler norm $|\cdot |$  in $\mathfrak g$, let 
$$
|x+iy|_T=\sup_{t\in [0,2\pi]} |x\cos t-y\sin t|
$$ 
be the Taylor Finsler norm of  $x+iy$ in the complexification $\mathfrak g^{\mathbb C}=\mathfrak g\oplus i\mathfrak g$ of $\mathfrak g$. This is a norm in the complexification that extends the original norm, and it is easy to check that $\|x+iy\|_T=\|x-iy\|_T$ (there are many possible complexifications, see \cite{munioz}). 

For a bounded linear operator $T:\g\to \g$, its complexification is defined as $T^{\mathbb C}(x+iy)=Tx+iTy$. It is easy to check that it is complex linear and bounded. The \emph{real spectrum} of $T$ is defined as 
\[
\sigma_{\mathbb R}(T)=\{t\in\mathbb R: T-t\textrm{ id }\; \textrm{ is not invertible}\}.
\]
An elementary fact that we will be using is that  $\sigma_{\mathbb R}(T)=\sigma(T^{\mathbb C})\cap \mathbb R$, and due to this we will drop both the superindex $\mathbb C$ for opearators, and the subindex $\mathbb R$ for the spectrum.

We claim that $|\cdot |_T$ is $\Ad_G$-invariant. For if $g\in G$ and $z=x+iy\in\mathfrak g^{\mathbb C}$, then
\begin{align*}
|\Ad_g(x+iy)|_T & =|\Ad_g(x)+i\Ad_g(y))|_T=\sup_{t\in [0,2\pi]} |\Ad_g(x)\cos t-\Ad_g(y)\sin t|\\
& =\sup_{t\in [0,2\pi]} |\Ad_g(x\cos t-y\sin t)|=\sup_{t\in [0,2\pi]} |x\cos t-y\sin t|  = |x+iy|_T.
\end{align*}
\end{rem}

\begin{defi}[Numerical range and Hermitian operators]
\label{numr} Let $(X,|\cdot|)$ be complex complete normed space. For non-zero $z\in X$ let 
\[
N_z=\{\varphi\in X^*: \|\varphi\|=1, \quad \varphi(z)=|z|\}.
\]
This is a nonempty closed convex set, by the Hahn-Banach theorem. Let $A$ be a bounded linear operator in $X$, the \emph{spatial numerical range} of $A$ is the set $V(A)\subset \mathbb C$ defined as
\[
V(A)=\{\varphi(Az): \varphi\in N_z,\quad |z|=1\}.
\]
It is well-known that $V(A)$ is connected and that if $\sigma(A)$ is the spectrum of $A$, then 
\begin{equation}\label{spvsra}
\sigma(A)\subset  \overline{V(A)},
\end{equation}
the closure of $V(A)$ in $\mathbb C$ (see \cite{bd1} pages 88 and 102 for proofs). The \emph{spectral radius of $A$} is $\rho(A)=\max \{|\lambda|: \lambda\in\sigma(A)\}$. An operator $A$ on $X$ is \emph{Hermitian} if for any $z\in X$  we have $|e^{i s A}z|=|z|$ for all $s\in \mathbb R$. It is also well-known that for Hermitian operators,
\begin{enumerate}
\item $\mathrm{co}\,\sigma(A)=\overline{\mathrm{co}\, V(A)}$
\item $\rho(A)=\|A\|=\max\limits_{|x|=1}|Tx|$.
\end{enumerate}
Here $\mathrm{co}$ denotes the convex hull of a subset of $\mathbb C$. See \cite{bd1} pages 53 and 86 for the proof of the first assertion (the closed convex hull of the spatial numerical range is the intrinsic numerical range), and see \cite{sinclair} for the second assertion. 
\end{defi}

\begin{teo}\label{teonorm}If $\g$ admits an $\Ad$-invariant Finsler norm, then 
\begin{enumerate}
\item $\sigma(\ad x)\subset i\mathbb R$ for any $x\in \g$.
\item $\B$ is  negative semi-definite in $\g$.
\item Every nilpotent of $\g$ is central.
\item $Z(\g)=\g^{\perp}=\{x\in \g: \B(x,x)=0\}=\rad$.
\end{enumerate}
\end{teo}
\begin{proof}
By Remark \ref{norposta}, $\g$ admits an actual $\Ad_G$ invariant norm, which we denote for this proof also as $|\cdot|$.  Let $v\in\g$, considered as an element of $\g^{\mathbb C}$. Let $\varphi$ be any unit norm linear functional on $\g^{\mathbb C}$ such that $\varphi(v)=|v|_T=|v|$. Let $A$ be the complexification of $\ad x$. Repeating the proof of Lemma \ref{desilie}, we obtain that for any $z\in \g^{\mathbb C}$, we have $\re\varphi (\ad x (z))=0$, which shows that $V(A)\subset i\mathbb R$, hence the first claim follows from \eqref{spvsra}. 

Since $\sigma(\ad x)\subset i\mathbb R$ for any $x\in \g$, then $\sigma(\ad^2 x)\subset (-\infty,0]$ and then $\B(x,x)=\tr(\ad^2x)\in (-\infty,0]$ showing that $\B$ is negative semi-definite.

Now let $x\in \g$. By the $\Ad_g$ invariance of the complexified norm, we get that $i\ad x$ is an Hermitian operator of $(\g^{\mathbb C},|\cdot|_T$), and in particular $\|\ad x\|=\|i\ad x\|=\rho(\ad x)$ by the results recalled before this theorem. If $\tr(\ad^2 x)=\B(x,x)=0$, and since the spectrum of $\ad^2 x$ is non-positive, then it must be that $\sigma( \ad^2 x)=\{0\}$, thus $\sigma(\ad x)=\{0\}$. But then $\|\ad x\|=\rho(\ad x)=0$, and this shows that $\ad x=0$ i.e. $x\in Z(\g)$. Since the inclusion $\g^{\perp}\subset \{x: \B(x,x)=0\}$ is trivial, from \eqref{radic} and what we just proved, we have so far that
\[
Z(\g)=\g^{\perp}=\{x\in \g: \B(x,x)=0\}\subset \rad.
\]
Now assume that $x\in \rad$, $x\notin Z(\g)$. Then it must be that $\sigma(\ad x)\ne \{0\}$ by the previous discussion or equivalently, there exists $r\ne 0$ and $y\ne 0$ such that 
\[
[x,[x,y]]=\ad^2 x (y)=-r^2 y.
\]
Being and ideal, we have that $[x,y]\in \rad$ also. But $[\rad ,\rad]\subset \g^{\perp}$  \cite[Chapter I, $\S$5.5]{bou1} hence $[x,[x,y]]\in \g^{\perp}=Z(\g)$. Thus $y\in Z(\g)$ and $[x,y]=0$ implying $y=0$, a contradiction. 
\end{proof}

\begin{teo}If $\g$ admits an $\Ad$-invariant Finsler norm, then $\g=Z(\g)\oplus \s$ where $\s$ is a compact semi-simple Lie algebra ($\B(x,x)<0$ for all $0\ne x\in \s)$. Moreover, $G$ is the direct product of the connected commutative group $H=\exp(Z(\g))$ and the compact connected group with finite center $K=\exp(\s)$ (hence $K$ is semi-simple).
\end{teo}
\begin{proof}
By the previous theorem, the results in  \cite[Chapter I, $\S$6.4]{bou1} and in $\S$6.8 loc. cit., we have that $\g$ is a reductive Lie algebra, and we obtain the assertions for the Lie algebra from the Levi-Malcev theorem. Now since $Z(\g)\oplus \s \ni z+s\mapsto e^se^z$ is a local diffeomorphism around $1$ and $G$ is connected, we have the direct product decomposition of $G=H\times K$. Since the Lie algebra of $K$ is compact, the group $K$ is compact with finite center by H. Weyl's theorem (\cite{bou3} Chapter IX, $\S$1.4).
\end{proof}

\begin{rem}[Abstract groups and orthogonal groups]\label{remamil} Assume $\g$ admits an $\Ad$-invariant Finsler norm. Pick a basis $\{e_i\}_{i=1,\dots,n}$ of $Z(\g)$ and define and inner product $\langle\cdot ,\cdot \rangle_{Z}$ there by declaring it an orthonormal basis. Then for $x_i=z_i+s_i\in \g=Z(\g)\oplus \s$ and $i=1,2$ we define
\[
\langle x_1,x_2\rangle_{\g}= \langle z_1,z_2\rangle_{Z}-\B(s_1,s_2)= \langle z_1,z_2\rangle_{Z}+\langle s_1,s_2\rangle_{\s},
\]
obtaining an inner product in $\g$. This inner product makes the direct sum into an orthogonal sum, and it is an $\Ad_G$-invariant inner product, from which a left-invariant Riemannian metric can be propagated to $G$. The induced distance in $G$, defined as the infima of the lengths of paths is bi-invarant. We note that
\begin{enumerate}
\item Milnor's results \cite{milcur} apply, in particular his Lemma 7.5: $G$ is isomorphic to the product of a compact group and an additive vector group. Moreover, when given $G$ the left-invariant Riemannian metric induced by any $\Ad_G$ invariant inner product, the geodesics are one-parameter groups (this follows from Milnor's formula for the covariant derivative $\nabla_x=\frac{1}{2}\ad x$, see page \cite[page 323]{milcur}) and the sectional curvature is always non-negative (Corollary 1.4, loc. cit.).
\item The adjoint representation of $\s$ is faithful, and by the negative definiteness of the Killing form we see that we can identify $\s$ with a Lie subalgebra of $\mathfrak{so}_n(\mathbb R)$, where $n=dim(\s)$, and $\mathfrak{so}_n(\mathbb R)$ in turn is a Lie subalgebra of $\su$.
\end{enumerate}
\end{rem}

\begin{rem}In Milnor \cite{milcur} Lemma 7.5, the decomposition $G=\mathbb R^m\times K'$ is obtained by looking at the universal covering $\pi:\widetilde{G}\to G$, which must be isomorphic to $\mathbb R^m\simeq Z(\g)$  and a compact group $K'$ with $Lie(K')=\s$. Then if $\Pi=\ker \pi$, and $pr_1:\widetilde{G}\to \mathbb R^m$ is the projection onto the abelian factor, he considers the linear span $V=span\{ pr_1(\Pi)\}$ and $V^{\perp}\subset \mathbb R^m$ its orthogonal there. It follows that $G\simeq    V^{\perp}\times \frac{(K'\times V)}{\Pi}= \mathbb R^n\times K''$, with $K''$ compact. It is fair to ask which of the decompositions (this one and the one we gave in the previous theorem) is more suitable for each case. 

Note for instance that if $G=S^1\times SO(n)$ which has $\g=\mathbb R\times \mathfrak{so}_n$, then the former describes better the situation  ($H=S^1, K=SO(n)$), while $\widetilde{G}=\mathbb R\times Spin(n)$,  $\Pi=\mathbb Z\times \mathbb Z_2$, $V^{\perp}=\{0\}$, therefore the later would give $K''=SO(n)\times S^1=G$ (and the abelian factor in Milnor's presentation is trivial i.e. $m=0$).
\end{rem}

\subsection{Distances and geodesics in the metric space setting}

By the theorem of Montgomery and Zippin about Hilbert's fifth problem (see \cite{montgomery}), a connected, locally compact, locally contractible topological group, if admits a Gleason distance, is in fact a Lie group. Exploiting this fact is the following application to a result by Berestovskii that we now apply to our situation in the following form:

\begin{teo}If $(G,\di)$ is a connected, locally compact, locally contractible topological group with a bi-invariant intrinsic distance, then $G=H\times K$ is the product of an abelian connected Lie group $H$ and a connected compact Lie group $K$ with discrete center. Moreover, the distance $\di$ is the left-invariant metric that comes from the $\Ad_G$-invariant Finsler norm
\begin{equation}\label{distfromdist}
|v|=\lim\limits_{t\to 0^+}\frac{\di(1,e^{tv})}{t},
\end{equation}
where $e=\exp$ is the exponential map of the Lie group $G$. The distance is reversible if and only if the Finsler norm is a norm.
\end{teo}
\begin{proof}
Adapting Berestovskii's Theorem 7 in \cite{beres}, to the case of non-reversible metrics (this is straightforward), we obtain that if $G$ is locally compact, locally contractible, and admits a bi-invariant intrinsic distance, then $G$ is a Lie group with 
 an $\Ad$-invariant Finsler metric. Now the first assertion follows from our previous theorem. Regarding the formula for the Finsler metric, this was also proved in \cite{beres}, Lemma 9 (taking into account the possible non-reversibility of the distance, the limit must be taken for positive $t$). 
\end{proof}

\begin{coro}\label{corodimenor}Let $(G,\di)$ be in the previous theorem. Let $B$ be an open ball around $0\in \g$ in the Finsler norm \eqref{distfromdist} such that $\exp|_B$ is a diffeomorphism onto its image $V\ni 1$. Then $\di(h,gh)=|\exp^{-1}h|$ for any $h\in V$ and any $g\in G$. In particular the path $t\mapsto ge^{tv}$ is minimizing for $t\in [0,1]$ as long as $v\in B$, and
\[
\di(e^{ty},e^{tx})=\di(1,e^{ty}e^{-tx})=\di(1,e^{B(ty,-tx)})=|B(ty,-tx)|
\]
for any $x,y$ and sufficiently small $t$ (here $B(v,w)=v+w+\frac{1}{2}[v,w]+\cdots$ is the Baker-Campbell-Hausdorff formula). Moreover
\begin{equation}\label{divsnorm}
\di(e^v,e^w)\le |v-w|\qquad \forall \, v,w\in \g.
\end{equation}
and if $[v,w]=0$ and $w-v\in B$, then we obtain an equality. If the norm is strictly convex and equality holds, it must be $[v,w]=0$.
\end{coro}
\begin{proof}
The assertions about distance and minimizing paths follow from the previous theorem and \cite[Corollary 4.12]{lar19}. The assertions about the inequality are proved in Theorem 4.17, loc. cit.
\end{proof}

\begin{defi}[Segments] Such paths $\delta(t)=ge^{tv}$ are called \emph{segments}, by the previous theorem 
\[
\di(\delta(s),\delta(t))=\di(ge^{sv},e^{tv})=\di(1,e^{(t-s)v})=(t-s)|v|
\]
as long as $0\le t-s\le 1$ (the first inequality is due to the possible non-reversibility of the distance). If the distance is reversible then this holds for $-1\le t-s\le 1$. Length minimizing paths are occasionally referred as \emph{short paths} or also \emph{metric geodesics}.
\end{defi}

So regarding geodesics and distance, the general situation is very much like in the case of a Riemannian left-invariant metric (Remark \ref{remamil}.1).

\begin{problem}There exists a neighbourhood $B$ of $0\in\g$ such that if $\Gamma:[a,b]\to (\g,|\cdot|)$ is a short path joining $0,x\in B$, then $\gamma=e^{\Gamma}$ is a short path in $(G,\di)$ joining $1,e^x$. Moreover for every $\varphi$ norming $x$ we have $\gamma_t^{-1}\gamma_t'\subset F_{\varphi}$ (this was proved in a more general setting in \cite{lar19} Section 4). Is there a short path $\gamma=e^{\Gamma}$ in $G$, such that $\Gamma$ is \emph{not} short in  $\g$?
\end{problem}

\begin{teo}\label{distanciavsmodulo} Let $x,y\in B\subset\g$ where $\exp$ is diffeomorphism, then 
\[
\di(e^x,e^y)\ge \frac{2}{\pi}|y-x|.
\]
\end{teo}
\begin{proof}
Let $B\subset \g$ be an open ball such that $\exp|_B:B\to V=\exp(V)$ is a diffeomorphism, and recall the formula for the differential of the exponential map in Lie groups:
\[
D\exp_{v}(w)=L_{e^v}\int_0^1 \Ad_{e^{-sv}}w\,ds = L_{e^v} \int_0^1 e^{-s\, \ad v}w\,ds=L_{e^v} F(\ad v)w
\]
where $F$ is the holomorphic map $F(\lambda)=\frac{1-e^{-\lambda}}{\lambda}$. Recall that the spectrum of $\ad v$ is inside the interval $i(-\|\ad v\|,\|\ad v\|)$ (Definition \ref{numr} and Theorem \ref{teonorm}). Then it must be $\|\ad v\|<\pi$ for all $v\in B$ otherwise we could find $v_0=t_0v$ such that $0\in \sigma(F(\ad v))=F(\sigma(\ad v)$, contradicting the invertibility of $D\exp_v$. Shrink a bit the ball considering $v\in rB$, with $0<r<1$, then $\|\ad v\|\le M(r)<\pi$ for all $v\in rB$. Since $|\ad v(x)|\le \|\ad v\||w|$, and with the same proof as in \cite[Proposition 4.6]{alr} we obtain
\[
|F(\ad v)^{-1}w|\le f(\|\ad v\|)|w|<f(M(r)) |w|
\]
for $f(t)=\frac{t/2}{\sin(t/2)}$. Let $\phi:\exp(rB)\to rB$ be the inverse of the exponential map, from $e^{\phi(g)}=g$ for $g\in \exp(rB)$, differentiating we obtain
\[
D\exp_{\phi(g)}D\phi_g\dot{g}=\dot{g}.
\]
Thus if $g=e^v$, and $\dot{g}=L_gw\in T_gG$, we have 
\[
L_{e^v} F(\ad v) D\phi_{e^v} (\dot{g})=D\exp_v D\phi_{e^v}(\dot{g})=\dot{g}=L_{e^v}w,
\]
hence $D\phi_{e^v}(\dot{g})=F(\ad v)^{-1}w$. Then
\[
|D\phi_g(\dot{g})|=|F(\ad v)^{-1}w|\le f(M(r))||w|=f(M(r)) |\dot{g}|_g\le \frac{2}{\pi}|\dot{g}|_g
\]
If $x,y\in B$, we have $x,y\in rB$ for some $r<1$, and from the previous inequality for the exponential chart it can be proved that
\[
|y-x|=|\phi(e^x)-\phi(e^y)|\le\frac{2}{\pi} \di(e^x,e^y)
\]
(see \cite[Sec 9.1.2]{largeo} or the proof of \cite[Prop. 12.22]{upmeier}). 
\end{proof}

We finish this section with a small lemma that will be really useful later in our calculus of curvatures.

\begin{rem}\label{bch}
The first terms of the Baker-Campbell-Hausdorff series are
$$
\exp^{-1}(e^{tx}e^{ty})=B(tx,ty)=t(x+y)+\frac{t^2}{2}[x,y]+\frac{t^3}{12}\left( [x,[x,y]]+[y,[y,x]]\right)+O(t^4).
$$
Then for small $r$ and any $t\in [0,1]$ we can write:
$$
\di(e^{rx},e^{ry})=\di(1,e^{(t-1)rx}e^{ry}e^{-rt x})=\di(1,e^We^{-trx})=\di(1,e^{Z_r})=|Z_r|,$$
where 
\begin{equation}\label{dobleve}
W=B((t-1)rx,ry)\quad\textrm{ and }\quad Z_r=B(W,-rt x).
\end{equation}
For given $t\in\mathbb R$, these expressions are well-defined for sufficiently small values of $r\in \mathbb R$.
\end{rem}

\begin{lem}\label{lemaB}With the notation of the last remark, for small $r$ we have:
\[
\di(e^{rx},e^{ry})=|r(y-x)+\frac{r^2}{2}(2t-1)[x,y]+\frac{r^3}{12}(6t^2-6t+1)[x,[x,y]]+\frac{r^3}{12}[y,[x,y]]+ O(r^4)|.
\]
In particular, for $t=1/2$ we obtain
\begin{equation}\label{dibch}
\di(e^{rx},e^{ry})=|r(y-x)+\frac{1}{48}r^3[x+y,[x,y]]+\frac{1}{16}r^3[y-x,[x,y]]+O(r^4)|.
\end{equation}
\end{lem}
\begin{proof}
Let's start by computing an expression for $W$ in (\ref{dobleve}):
\begin{align*}
W&=BCH((t-1)rx,ry)\\
&=(t-1)rx+ry+\frac{1}{2}[(t-1)rx,ry]+\frac{1}{12}[(t-1)rx-ry,[(t-1)rx,ry]]+O(r^4)\\
&=r\left((t-1)x+y\right)+\frac{r^2}{2}(t-1)[x,y]+\frac{r^3}{12}(t-1)^2[x,[x,y]]-\frac{r^3}{12}(t-1)[y,[x,y]]+O(r^4).
\end{align*}
Now, we write
$$B(W,-rtx)=A+B+\frac{1}{2}[A,B]+\frac{1}{12}[A-B,[A,B]]+O(r^4).
$$
We compute each term separately, starting with $A+B$
\begin{align*}
A+B= r(y-x)+\frac{r^2}{2}(t-1)[x,y]+\frac{r^3}{12}(t-1)^2[x,[x,y]]-\frac{r^3}{12}(t-1)[y,[x,y]]+O(r^4).
\end{align*} 
To compute $[A,B]$, we write:
\begin{align*}
[A,B]&= [r((t-1)x+y)+\frac{r^2}{2}(t-1)[x,y],-rtx]+O(r^4)\\
&=-r^2[(t-1)x+y,tx]-\frac{r^3}{2}(t-1)[[x,y],-rtx]+O(r^4)\\
&=tr^2[x,y]+\frac{r^3}{2}(t^2-t)[x,[x,y]]+O(r^4).
\end{align*}
To compute $[A-B,[A,B]]$, we write
\begin{align*}
[A-B,[A,B]]&=[rtx-rx+ry+rtx,  tr^2[x,y]+ O(r^3)  ]+O(r^4)\\
&=r^3[2tx-x+y,t [x,y]]+O(r^4)\\
&=tr^3[2tx-x+y,[x,y]]+ O(r^4)\\
&=(2t^2-t)r^3[x,[x,y]]+tr^3[y,[x,y]]+O(r^4).
\end{align*}
Finally, we can compute 
\begin{align*}
B(W,-rtx)&= r(y-x)+\frac{r^2}{2}(2t-1)[x,y]+\alpha r^3[x,[x,y]]+\beta r^3 [y,[x,y]],
\end{align*}
where
$$
\alpha=\frac{1}{12}(t-1)^2+\frac{1}{4}(t^2-t)+\frac{1}{12}(2t^2-t)=\frac{1}{12}(6t^2-6t+1),
$$
and
\[
\beta=-\frac{1}{12}(t-1)+\frac{1}{12}t=\frac{1}{12}.
\]
In conclusion we have 
\[
Z_r=r(y-x)+\frac{r^2}{2}(2t-1)[x,y]+(\frac{t^2}{2}-\frac{t}{2}+\frac{1}{12})r^3[x,[x,y]]+(-\frac{t^2}{12}+\frac{t}{4}-\frac{1}{12})r^3[y,[x,y]].\qedhere
\]
\end{proof}

\section{Curvature}

Our notions and definitions are motivated by the following remark by J. Milnor in his \textit{Morse Theore} lecture notes \cite[ p.101]{milnorm}: \emph{consider an observer at $p$ looking in the direction of the unit vector $U$ towards a point $q=\exp_p(rU)$. A small line segment at $q$ with length $L$, pointed in a direction corresponding to the unit vector $V\in T_pM$, would appear to the observer as the line segment of length}
$$
L(1+\frac{s^2}{6}\langle R(U,V)U,V\rangle+  o(s^3)),
$$
more precisely: 
$$\langle R_p(x,y)y,x \rangle_p=6 \|y-x\|_p^2\lim_{r\to 0^+}\frac{r||y-x||_p-d(\exp_p(rx),\exp_p(ry))}{r^2d(\exp_p(rx),\exp_p(ry))},  $$
where $R$ is the curvature tensor of the Riemannian metric considered. This approach was taken in \cite{cocoeste} and \cite{cl}, though it was in the setting of Hermitian matrices, the tangent space to the manifold of positive invertible matrices. We will now prove some additional properties about $\Ad$-invariant Finsler norms, their geometry and norming functionals. We will postpone the presentation of the actual metric curvature to Section \ref{sec:s}.

\subsection{Faces of the sphere, smoothness and convexity}

The notion of norming functional will be key to many descriptions of the geometry of $G$ and $\g$. Let $|\cdot|$ be a Finsler norm in $\g$.

\begin{defi}[Norming functionals]\label{dualnorm} For $\varphi\in \g'$ (the dual space of $\g$) consider 
\[
\|\varphi\|=\max\{\varphi(x): |x|\le 1\}.
\]
This defines a Finsler norm in the dual space. We say that $\varphi\in \g'$ \emph{norms} $v\in \g$ if $\varphi$ has unit norm and $\varphi(x)=|x|$.  For each $\lambda\in\mathbb R_{>0}$, if $\varphi(V)=|V|$ and $\|\varphi\|=1$ we have $\varphi(\lambda V)=\lambda |V|=|\lambda V|$ thus $\varphi$ norms the whole ray $\lambda V$, $\lambda >0$. In particular $N_{\lambda v}=N_v$ for any $\lambda>0$.
\end{defi}

Now we recall the notions of extreme points related to strict convexity of a sphere of the norm, and the notion of face of the sphere.

\begin{defi}[Extreme points]\label{strictc}
Let $\overline{B_1}$ be the closed unit ball of the norm. Being a compact convex set it is by the Krein-Milman theorem the convex hull of its extreme points. For $0\ne x\in \g$, we say that $x$ is \textit{extreme} if $x/|x|$ is an extreme point of $\overline{B_1}$, equivalently we say that the norm is \emph{strictly convex} in $x$.
\begin{enumerate}
\item The norm is strictly convex if and only if all the non-zero vectors are extremal. Equivalently, for any $x\ne 0 $ and any $\varphi$ norming $x$, it must be $F_{\varphi}=\{x\}$.
\item A norm is strictly convex at $x\ne 0$ if and only if there exists a norming functional $\varphi$ for $x$ such that $F_{\varphi}=\{x\}$.
\end{enumerate}
\end{defi}

\begin{defi}[Faces]\label{caras} A \textit{face} $S\subset \overline{B_r}$ of the normed space $(\g,|\cdot|)$ is a set such that every open segment $(x:y)\subset \overline{B_r}$ that meets $S$ is contained in $S$. In other words, they are the extremal subsets of the closed ball, and in particular the extremal points are the singleton faces.

An \textit{exposed face} $F$ of the ball $B_r$  is the intersection of the closed ball $\overline{B_r}$ with the hyperplane determined by a unit norm functional $\varphi\in \g', \|\varphi\|=1$, i.e.
$$
F_{\varphi}(r)=\overline{B_r}\cap\{x\in\g:\varphi(x)=r\}.
$$
Any exposed face is a face, but not the other way around. We will usually omit the number $r$ and $F_{\varphi}$ will refer to the face containing a certain vector $v$, thus $r=|v|$.

The \textit{cone generated} by a exposed face $F_{\varphi}$ is $C_{\varphi}=\mathbb R_+ F_{\varphi}$. By the observation before these definitions, this cone consists exactly of those $x\in \g$ such that $\varphi(x)=|x|$ for this given unit norm $\varphi$.
\end{defi}

\begin{rem}[Convexity and smoothness]\label{convesmooth} For any Finsler norm $|\cdot|$ in a finite dimensional vector space $X$ we have:
\begin{enumerate}
\item The norm $|\cdot|$ is G\^{a}teaux differentiable at $x\ne 0$ if and only if it is Fr\'echet differentiable  (this follows from \v{S}mulian Lemma, see \cite[Lemma 8.4]{fabian} for instance). Therefore we simply say that the norm is \textit{smooth} when this happens for any $x\ne 0$; in that case the norm function is in fact $C^1$ away from $x=0$ \cite[Corollary 8.5]{fabian}.
\item The norm is smooth if and only if the dual norm is strictly convex if and only if there is a unique functional norming each $x\ne 0$ \cite[Lemma 8.4 and Fact 8.12]{fabian}. Another proof of this last assertion can be derived from 
\begin{equation}\label{derlat}
\lim\limits_{t\to 0^-}\frac{|x+ty|-|x|}{t}=\min\limits_{\varphi\in N_x}\varphi(y)\le \max\limits_{\varphi\in N_x}\varphi(y)=\lim\limits_{t\to 0^+}\frac{|x+ty|-|x|}{t}
\end{equation}
which was proved in \cite{larey} (see Proposition 3.3 and Remark 3.4 there).  We also remark that the increment on the right is a non-decreasing function, in fact, it gets smaller when $t\to 0^+$.  
\end{enumerate}
\end{rem}

\begin{rem}[Chain rule for subdifferentials]\label{chainrule} If $b:(-\varepsilon,\varepsilon)\to (X,|\cdot|)$, and $b(0)\ne 0$ then 
\[
\lim\limits_{s\to 0^+}\frac{|b(s)|-|b(0)|}{s}= \max\limits_{\varphi\in N_{b(0)}}\varphi(b'(0)),
\]
provided $b(s)=b(0)+sb'(0)+o(s)$ with $o(s)/s\to 0$ for $s\to 0^+$. Indeed, it is easy to check that this limit, minus the last one in \eqref{derlat}, goes to zero. By the previous remark we obtain the claim.
\end{rem}

\subsection{Compact semi-simple algebras}

\begin{rem}\label{lasum}If $G$ admits a bi-invariant distance, we have already shown that $\g=Z(\g)\oplus \s$ with $\s$ a compact semi-simple Lie algebra. It is clear from Corollary \ref{corodimenor} that if either $x$ or $y\in Z(\g)$ then $\di(e^{rx},e^{ry})=|ry-rx|$ and in particular $S(x,y)=0$  (Definition \ref{defS}  below). We want to study the condition $S(x,y)=0$ in more depth, and this requires some machinery.
\end{rem}

\begin{defi} For $v=v_0+v_1, w=w_0+w_1\in\g=Z(\g)\oplus\s$, we indicate with 
\[
\langle v,w\rangle=\langle v_0,w_0\rangle_Z-\B(v_1,w_1)=\langle v_0,w_0\rangle_Z-\tr(\ad v_1\circ \ad w_1)=\langle v_0,w_0\rangle_Z+(\ad v_1|\ad w_1)
\]
an $\Ad_G$-invariant inner product in $\g$ as in Remark \ref{remamil}, where its restriction to $\s$ is the opposite of the Killing form and $\s$ is orthogonal to $Z(\g)$. We denote with 
\[
\|v\|_F=\sqrt{\|v_0\|_Z^2+(\ad v_1|\ad v_1)}=\sqrt{\langle v,v\rangle}
\]
the extended Frobenius norm induced by this inner product. From \eqref{killcyc} we see that
\begin{equation}\label{cyclic}
\langle [x,y],z\rangle= \langle x, [y,z]\rangle =- \langle y, [x,z]\rangle
\end{equation}
for any $x,y,z\in \g$, and in particular  $\ad x:\g\to \g$ is skew-adjoint for this inner product; $\ad x$  has $\s$ as an invariant subspace and $\ad x$  is non-trivial for non-zero $x\in \s$. 
\end{defi}

\begin{rem}[Real root decomposition]\label{rootdecom} We collect here some known facts of compact semi-simple Lie algebras $\s$. Let $\ah\subset  \s$ be a Cartan subalgebra, let $\Delta$ be the set of (real) roots of $\s$ with respect to this Cartan subalgebra, and denote $\Delta_+$ the positive roots, each root $\alpha$ represented by a nonzero vector $h_{\alpha}\in \ah$. There is a set of vectors in $\s$ (the \emph{real root vectors}) 
\[
\{u_\alpha, v_{\alpha}: \alpha \in\Delta_+\}
\]
orthonormal with respect to the Killing form, such that for each $h\in \ah$ 
\begin{align}\label{weights}
[h,u_\alpha] & =\alpha(h)v_\alpha \qquad [h,v_\alpha]=-\alpha(h)u_\alpha \qquad [u_\alpha,v_\alpha]= h_\alpha
\end{align}
where $\alpha(\cdot)=\langle h_\alpha,\cdot\rangle$. The $\{h_{\alpha}\}_{\alpha\in\Delta_+}$ span the Cartan subalgebra over the real numbers, and $\beta(h_{\alpha})\in\mathbb Z\, |\alpha|^2$ for any $\alpha,\beta\in \Delta$, where $|\alpha|=\|h_{\alpha}\|_F$. We have 
\[
\s=\ah\oplus_{\alpha\in \Delta_+} Z_\alpha=\ah\oplus_{\alpha\in \Delta_+} \mathbb R u_\alpha\oplus_{\alpha\in \Delta_+} \mathbb R v_\alpha
\]
with orthogonal direct sums. Moreover for each $h\in \ah$, we have
\begin{equation}\label{diago}
\ad h=\sum_{\alpha\in\Delta_+} \alpha(h)(v_{\alpha}\otimes u_{\alpha}-u_{\alpha}\otimes v_{\alpha})= \sum_{\alpha\in\Delta_+} \alpha(h) T_{\alpha}
\end{equation}
where we write $T_{\alpha}=(v_{\alpha}\otimes u_{\alpha}-u_{\alpha}\otimes v_{\alpha})$ for short and $(x\otimes y)z=\langle z,y\rangle x$. Note that $T_{\alpha}T_{\beta}=0$ when $\alpha\ne \beta$, and moreover
\[
T_{\alpha}^2=-(u_{\alpha}\otimes u_{\alpha}+v_{\alpha}\otimes v_{\alpha})=-P_{\alpha}
\]
and $P_{\alpha}$ is a 2-dimensional orthogonal projection in $(\s,\langle \cdot,\cdot\rangle)$. Therefore $\ad^2h=-\sum_{\alpha}\alpha(h)^2 P_{\alpha}$ and the eigenvalues of $\ad h$ are $\{\pm i\alpha(h): \alpha\in \Delta_+\}$. 
\end{rem}

For a full exposition with proofs of the facts quoted in the  previous remark, see Appendix B in \cite{malk} or Knapp's book \cite[Theorem 6.11]{knapp}, where we picked  $h_{\alpha}=iH_{\alpha}$, and
\[
u_{\alpha}=\frac{1}{\sqrt{2}}(X_{\alpha}-X_{-\alpha})\qquad v_{\alpha}=\frac{1}{\sqrt{2}}i(X_{\alpha}+X_{-\alpha})
\]
in the notation of the cited book.

\begin{rem}[Norming functionals]\label{normif} Let $|\cdot|$ be an $\Ad$-invariant Finsler norm $\g$, let $\varphi\in\g'$ be a norming functional for $v\in\g$. Then by Riesz representation theorem for linear forms there exists a unique $z\in\g$ such that $\varphi(x)=\langle z,x\rangle$ for any $x\in\g$. If $g\in G$ and  $\psi=\langle \Ad_g z,\cdot\rangle$, then we claim that $\|\psi\|=\|\varphi\|$ (Definition \ref{dualnorm}): since $G$ acts by orthogonal transformations and then
\[
\psi(w)=\langle \Ad_g z,w\rangle=\langle z, \Ad_g^{-1}w\rangle= \varphi(\Ad_{g^{-1}}w)\le \|\varphi\| |\Ad_{g^{-1}}w|=\|\varphi\| |w|
\]
by the $\Ad$-invariance of the norm in $\g$, and with a similar reasoning we obtain the reversed inequality. Thus 
\begin{equation}\label{sonigual}
\|\varphi\circ \Ad_g \|=\|\varphi\| \quad \forall g\in G.
\end{equation}
Moreover we have $\varphi([v,x])=0$ and $\varphi([x,[x,v]])\le 0$ for any $x\in\g$ by  Lemma \ref{desilie}. And from
\[
0=\varphi ([x,v])=\langle z,[x,v]\rangle= -\langle [z,v],x\rangle
\]
picking $x=[z,v]$ we see that $\|[z,v]\|_F=0$ thus $[z,v]=0$ whenever $\varphi=\langle z,\cdot\rangle$ norms $v$.  Write $v=v_0+v_1,z=z_0+z_1\in Z(\g)\oplus\s$; then it must be $[z_1,v_1]=0$ and we can take a Cartan subalgebra containing $v_1,z_1$ (see $\S2$.Proposition 10 in Bourbaki's \cite{bou3}), and fix $\Delta_+$ a positive root system in $\s$.
\end{rem}

\begin{rem}[Norming functionals and Cartan subalgebras]\label{perm}
Let $z,v\in \s$ with $[v,z]=0$, fix a Cartan subalgebra $\ah\subset \s$ containing them. In this case from \eqref{diago} we see that
\[
\langle z,v\rangle=(\ad z|\ad v)=-\tr(\ad z\circ \ad  v)=\sum_{\alpha} \alpha(z)\alpha(v) \tr(P_{\alpha})=2\sum_{\alpha} \alpha(z)\alpha(v).
\]
Let $w\in\s$ and write $w=w_0+\sum_{\alpha} a_{\alpha}u_{\alpha}+b_{\alpha} v_{\alpha}$ with $w_0\in \ah$. We have
\[
[z,w]= \sum_{\alpha\in \sop(z)} \alpha(z)( a_{\alpha} v_{\alpha}  - b_{\alpha} u_{\alpha}) \qquad [w,v]= \sum_{\alpha\in \sop(v)} \alpha(v)( -a_{\alpha} v_{\alpha}  + b_{\alpha} u_{\alpha}).
\]
Thus from the previous lemma we have that, if $\varphi$ norms $v$
\begin{equation}\label{menorigual}
-\sum_{\alpha\in \sop(v)\cap \sop(z)} \alpha(z)\alpha(v)(a_{\alpha}^2+b_{\alpha}^2)=\langle [z,w], [w,v]\rangle=\varphi([w,[w,v]])\le 0.
\end{equation}
Picking $w=u_{\alpha}$ it follows that when $\varphi=\langle z,\cdot\rangle$ norms $v$, then
\begin{equation}\label{moig}
\alpha(z)\alpha(v)\ge 0\qquad \textrm{ for all } \alpha\in \Delta_+.
\end{equation}
\end{rem}

\begin{defi}Let $\varphi=\langle z,\cdot\rangle \in\s^*$ be a norming functional for $v\in\s$ as in Lemma \ref{normif}, and let $\ah$ be a Cartan subalgebra containing  $v,z$. Let $\Delta_+$ be a set of positive roots and let
\[
\sop(v)=\{ \alpha\in \Delta_+: \alpha(v)\ne 0\},
\]
be the roots supporting $v$, and likewise with $z$. Note that 
\[
\sop(v)^c=\{\alpha: \alpha(v)=0\}=\{\alpha: [v,u_{\alpha}]=0=[v,v_{\alpha}]\}
\]
by \eqref{weights}, and also that $v$ is regular (Definition \ref{adapted}) iff $\sop(v)=\Delta_+$. We will also consider the subspace $S_v=\bigoplus_{\alpha\in \sop(v)} Z_{\alpha}$ and we will denote as $P_v:\g\to S_v$ the orthogonal projection onto $S_v$.
\end{defi}

\begin{lem} Let $|\cdot|$ be an $\Ad$-invariant Finsler norm $\g$, let $\varphi=\langle z,\cdot\rangle \in\s^*$ be a norming functional for $v\in\g$ as in Remark \ref{normif}, let $F_{\varphi}\subset \g$ be the face supported by $\varphi$. Write $v=v_0+v_1,z=z_0+z_1\in Z(\g)\oplus\s$, let $\ah$ be a Cartan subalgebra containing $v_1,z_1$, let $\alpha\in\Delta_+$. Then
\begin{enumerate}
\item If $\alpha\in \sop(v_1)^c\cap \sop(z_1)$ then there exists $\psi\ne \varphi$ also norming $v$. 
\item If $\alpha\in \sop(v_1)\cap \sop(z_1)^c$ then $F_{\varphi}$ is not a singleton.
\end{enumerate}
\end{lem}
\begin{proof}
Let $\ker \alpha=\{h_{\alpha}\}^{\perp}\subset \ah$ and consider the decompositon
\[
\s= \mathbb R h_{\alpha} \oplus \ker\alpha \oplus Z_{\alpha} \oplus_{\beta\ne \alpha} Z_{\beta}
\]
where every direct summand is orthogonal with respect to the Killing form, and the first two terms sum up to $\ah$. Let $|\alpha|^2=\alpha(h_{\alpha})$ and write 
\begin{equation}\label{vz}
v=v_0+\frac{\alpha(v)}{|\alpha|^2} h_{\alpha}+v_{\perp}\qquad\qquad z=z_0+\frac{\alpha(z)}{|\alpha|^2} h_{\alpha}+ z_{\perp},
\end{equation}
where $h_{\alpha}\perp v_{\perp}\in \ah$ and likewise $h_{\alpha}\perp z_{\perp}\in \ah$.  Then
\[
|v|=\langle z,v\rangle =\langle v_0,z_0\rangle_Z+\frac{\alpha(v)\alpha(z)}{|\alpha|^2}+ \langle v_{\perp},z_{\perp}\rangle.
\]
By Lemma B.1 in \cite{malk} there exists $g\in K$ (the connected semisimple Lie group integrating $\s$) such that $z'=\Ad_g z_1$ verifies: a) $\alpha(z')=0$, b) $z_1$ and $z'$ have the same $\ker \alpha$ component, c) the components of $z_1,z'$ in  $\oplus_{\beta\ne \alpha} Z_{\beta}$ have the same norm. Then 
\begin{equation}\label{zprima}
z''=\Ad_g(z)=z_0+z'=z_0+0 + z_{\perp} + a_{\alpha}u_{\alpha}+b_{\alpha}v_{\alpha} +0\in Z(\g)\oplus \mathbb R h_{\alpha}\oplus \ker\alpha \oplus Z_{\alpha}\oplus \oplus_{\beta\ne \alpha}Z_{\beta},
\end{equation}
since the component of $z$ in $\oplus Z_{\alpha}$ was null. Note that if $\psi=\langle z'',\cdot\rangle$, then $\psi=\varphi\circ \Ad_{g^{-1}}$ thus $\|\psi\|=\|\varphi\|=1$ by \eqref{sonigual}, and
\[
\psi(v)=\langle z'',v\rangle =\langle z_0,v_0\rangle_Z+ 0 + \langle v_{\perp},z_{\perp}\rangle= \langle z,v\rangle =|v|,
\]
because $\alpha(v)=0$. In synthesis, $\psi$ also norms $v$. Since $\alpha(z)\ne 0$, we have $z\ne z''$ and then $\psi\ne \varphi$.  Now we can do the same, but with $v$: there exists $g\in K$ such that, if  $v'=\Ad_g v$ (hence $|v'|=|v|$) it verifies
\[
v'=v_0+0 + v_1 + x_{\alpha}u_{\alpha}  +y_{\alpha}v_{\alpha} + 0.
\]
Hence 
\[
\varphi(v')=\langle z,v'\rangle= \langle z_0,v_0\rangle_Z+\langle v_{\perp},z_{\perp}\rangle=|v|=|v'|
\]
provided $\alpha(z)=0$, showing that $\{v,v'\}\subset F_{\varphi}$. Since $\alpha(v)\ne 0$, we have $v\ne v'$ and this proves the second assertion.
\end{proof}

\begin{defi}\label{adapted}
We say that $\varphi=\langle z=z_0+z_1,\cdot \rangle$ norming $v=v_0+v_1$ is \emph{adapted} to $v$ if: there exists a Cartan subalgebra $\ah$ with $v_1,z_1\in \ah$ and  positive root system $\Delta_+$, such that  for each $\alpha\in \Delta_+$ we have that $\alpha(v_1)=0$ implies $\alpha(z_1)=0$.  

In particular, if $v_1$ is \emph{regular} (its centralizer $\ah$ has minimal dimension among abelian subalgebras, and it is then a Cartan subalgebra $\ah$), then there is only one Cartan subalgebra containing $v_1$. It can be shown that $v_1$ is regular if and only if $\alpha(v_1)\ne 0$ for all $\alpha\in \Delta_+$. Then any norming $\varphi=\langle z,\cdot\rangle$ with $z_1\in\ah$ is adapted to $v$.
\end{defi}

\begin{rem}
We remark that for $U(n)$ we established the existence of norming functionals adapted to a vector $v$ by permutating the elements of the basis of $\mathbb C^n$ \cite[Lemma 2.38]{larey}. With the same idea, the proof can be extended to $SU(n)$ and $O(n)$. However, since in general the group $G$ might not act transitively on the positive roots (roots might even have different lengths), this mechanism of proof is not suitable for generalization. In what follows we establish in general the existence of norming functionals adapted to $v$, but with an entirely different mechanism of proof.
\end{rem}

\begin{lem}\label{adaptada}For each $v\ne 0$ in $\g$ there exists at least one $\varphi_v=\langle z_v ,\cdot\rangle$ norming $v$ such that $\varphi_v$ is adapted to $v$. If the norm is smooth at $v$ then the unique functional norming $v$ is adapted to $v$.
\end{lem}
\begin{proof}
Consider the set $C_v=\{w\in \g: \varphi=\langle w,\cdot\rangle \; \textrm{ norms } v\}$. Then $C_v\subset \g$ is compact, convex, and non-empty. Let $z_v\in C_v$ be an element of minimal Frobenius norm, i.e. $\|z_v\|_F=\sqrt{\langle z_v,z_v\rangle}  \le \|w\|_F$ for all $w\in C_v$. This element exists because $C_v$ is compact and convex, it's non-zero because $z_v$ norms $v$, and it's unique in $C_v$ since the Frobenius norm is strictly convex: if there exists two of them $z_v,z_v'\in C_v$ of minimal Frobenius norm, then their arithmetic mean is also norming for $v$ and has strictly smaller Frobenius norm than $z_v$. Write $v=v_0+v_1, z_v=z_0+z_1\in Z(\g)\oplus \s$. Let $\ah$ be a Cartan subalgebra containing $v_1,z_1$, let $\Delta_+$ be a positive root system with respect to this Cartan subalgebra. We claim that $\varphi=\langle z_v,\cdot\rangle$ is adapted to $v$. Take $\alpha\in \Delta_+$ such that $\alpha(v_1)=0$, write $v,z_v$ as in the previous lemma \eqref{vz};  then $v_1=v_{\perp}$ and we claim that $z_1=z_{\perp}$. If not, we have  $\alpha\in \sop(v_1)^c\cap \sop(z_1)$ and there exists $g\in G$ such that 
\[
\Ad_g(z_v)=z_0+z_{\perp}+ w_{\alpha}
\]
as in the proof of that lemma, with $w_{\alpha}\in Z_{\alpha}$. Let $\psi=\langle \Ad_g z,\cdot\rangle$, then $\psi$  also norms $v$. But then $\|\Ad_g z_v\|_F=\|z_v\|_F$ hence it must be
\[
z_0+ z_{\perp}+ w_{\alpha}=\Ad_g z_v=z_v= z_0+\frac{\alpha(z_1)}{|\alpha|^2}+ z_{\perp}.
\]
This is only possible if $w_{\alpha}=\alpha(z_1)=0$, a contradiction. This proves the existence. Now if the norm is smooth, the set $C_v$ is a singleton and it must be $C_v=\{z_v\}$ therefore the unique $\varphi$ norming $v$ is adapted to $v$.
\end{proof}

\begin{teo}\label{teoNbis} Let $v_0+v_1,z=z_0+z_1\in \g=Z(\g)\oplus\s$, let $\varphi=\langle z,\cdot\rangle$ norming $v$, let $\ah$ be a Cartan subalgebra containing $v_1,z_1$. Then  $\varphi([x,[x,v]])=0$ if and only if $[P_{v_1}x,z]=0$. In this case
\begin{enumerate}
\item If $\varphi$ is adapted to $v$, then $[x,z]=0$.
\item If there exists a unique functional norming $v$, then $[x,z]=0$.
\item If $F_{\varphi}=\{v\}$ then $[x,v]=0$. 
\end{enumerate}
\end{teo}
\begin{proof}
Since all the conditions are of the form ``$x$ commutes with $y$'', we might as well assume that $x\in \s$. First note that if $x=x_k+ \sum_{\alpha\in \Delta_+}x_{\alpha} u_{\alpha}+y_{\alpha}v_{\alpha}$ with $x_k\in\ah$, then we obtain $P_{v_1}x$ by dropping $x_0$ and all the roots $\alpha$ such that $\alpha(v)=0$. We extend the roots $\alpha$ as $0$ to $Z(\g)$ by orthogonality to avoid overloading the notation. Then $\alpha(v)=\alpha(v_1)$ and likewise with $z$. Now note that 
\[
[P_{v_1}x,z]=\sum_{\alpha\in \sop(v)} x_{\alpha}[u_{\alpha},z]+y_{\alpha}[v_{\alpha},z]= \sum_{\alpha\in \sop(v)} \alpha(z)( y_{\alpha} u_{\alpha}  - x_{\alpha} v_{\alpha})
\]
by \eqref{weights}, and then
\begin{equation}\label{pxz}
\|[P_{v_1}x,z]\|_F^2=\sum_{\alpha\in \sop(v)} \alpha(z)^2( x_{\alpha}^2+ y_{\alpha}^2),
\end{equation}
so $[P_{v_1} x,z]=0$ if and only if  $\alpha(z)(x_{\alpha}^2+y_{\alpha}^2)=0$ for all $\alpha\in \sop(v)$. By equation \eqref{menorigual}, we have
\[
-\sum_{\alpha\in \sop(v)\cap \sop(z)} \alpha(z)\alpha(v)(x_{\alpha}^2+y_{\alpha}^2)=\varphi([x,[x,v]]).
\]
Then if $[P_v x,z]=0$, the whole sum is equal to $0$. We have proved that $[P_v x,z]=0$ implies $\varphi([x,[x,v]])=0$.  If
\[
0=\varphi([x,[x,v]])= \sum_{\alpha} \alpha(z)\alpha(v)(x_{\alpha}^2+y_{\alpha}^2),
\]
and since $\alpha(v)\alpha(z)\ge 0$ for all $\alpha$ by equation \ref{moig}, it must be
\begin{equation}\label{laposta}
\alpha(v)\alpha(z)(x_{\alpha}^2+y_{\alpha}^2)=0 \qquad \forall\, \alpha\in \Delta_+.
\end{equation}
For those $\alpha\in\sop(v)$ we can cancel out $\alpha(v)$ and we have 
\[
\alpha(z)(x_{\alpha}^2+y_{\alpha}^2)=0 \qquad \forall\, \alpha\in \sop(v),
\]
and looking at \eqref{pxz}, this proves that $[P_v x,z]=0$. The seconds assertion follows immediatly from the previous lemma, since if there's only one norming funcional for $v$ it must be adapted.

Now asume that $F_{\varphi}=\{v\}$. Since  
\[
\|[x,v]\|_F^2=\sum_{\alpha\in \sop(v)}\alpha(v)^2(x_{\alpha}^2+y_{\alpha}^2),
\]
Let $\alpha\in \Delta_+$. If $\alpha(v)=0$, this term vanishes from the sum. If $\alpha(v)\ne 0$ we must have $\alpha(z)\ne 0$. Then from equation \eqref{laposta} we see that $x_{\alpha}^2+y_\alpha^2=0$ then this term also vanishes from the sum, and we conclude that $[x,v]=0$.
\end{proof}

\subsection{The form $S$ and sectional curvature}\label{sec:s}

Now we introduce the sectional curvature of a pair of tangent vectors at the identity; by the bi-invariance of the metric this is in fact a quantity defined for a pair of tangent vectors at any point $g\in G$. To make Milnor's approach  more precise, we begin with a definition:
\begin{defi}[The form $S$]\label{defS} Let $x,y\in \g$, let
\[
S(x,y)=6|y-x|^2\lim\limits_{r\to 0^+} \frac{r|y-x|-\di(e^{rx},e^{ry})}{r^2\di(e^{rx},e^{ry})}.
\]
\end{defi}

\begin{rem}\label{laf} For $x\ne y$, we show below that the limit exists, but we can make a simplification before that: we claim that
\[
S(x,y)=6|y-x|\lim\limits_{r\to 0^*}\frac{r|y-x|-\di(e^{rx},e^{ry})}{r^3}.
\]
This is because by Lemma \ref{lemaB}, we have $\lim\limits_{r\to 0^+}\frac{ \di(e^{rx},e^{ry})}{r}=|y-x|$. From Remark \ref{bch} we have
\[
\di(e^{rx},e^{ry})=\di(1,e^{ry},e^{-rx})=|B(ry,-rx)|
\]
for small $r>0$, thus $[x,y]=0$ implies $S(x,y)=0$, i.e.  the plane containing $x,y$ is \emph{flat}.
\end{rem}

\begin{prop}[Sectional curvature along two rays]\label{ese} Let $x,y\in \g$, then 
\begin{align*}
S(x,y) & =6|y-x|\lim\limits_{s\to 0^+} \frac{|y-x|-|(y-x)+\frac{1}{48}s[x+y,[x,y]]|}{s}\\
&=-\frac{|y-x|}{4}\max\limits_{\varphi\in N_{y-x}}\varphi([y,[y,y-x]])=-\frac{|y-x|}{4}\max\limits_{\varphi\in N_{y-x}}\varphi([x,[x,y-x]])\geq 0.
\end{align*}
\end{prop}
\begin{proof}
Note that since $d(e^{rx},e^{ry})\leq |rx-ry|$ for $r$ small enough (Theorem \ref{distanciavsmodulo}), the limit is non-negative. Now we compute the limit using Lemma \ref{lemaB}, taking $t=1/2$: for small $r>0$ we have
\[
\frac{1}{r}\di(e^{rx},e^{ry})= |y-x+\frac{1}{48}r^2[x+y,[x,y]]+\frac{1}{16}r^2[y-x,[x,y]]+O(r^3)|=|b(r^2)|
\]
for 
\begin{equation}\label{labe}
b(s)=y-x+\frac{1}{48}s[x+y,[x,y]]+\frac{1}{16}s[y-x,[x,y]]+O(s^{3/2})=b(0)+sb'(0)+o(s).
\end{equation}
Then we have
\begin{align*}
S(x,y)&=6|y-x| \lim\limits_{r\to 0^+} \frac{r|y-x|-d(e^{rx},e^{ry})}{r^3}= -6 |y-x|\lim\limits_{r\to 0^+} \frac{|b(r^2)|-|y-x|}{r^2}\\
& = - 6 |y-x| \lim_{s\to 0^+}  \frac{|b(s)|-|b(0)|}{s}= -6 |y-x|\max\limits_{\varphi\in N_{y-x}}\varphi(b'(0))\\
&= -6|y-x|\max\limits_{\varphi\in N_{y-x}}\varphi(\frac{1}{48}[x+y,[x,y]]+ \frac{1}{16}[y-x,[x,y]])
\end{align*}
by Remark \ref{chainrule}. Now, if $\varphi(y-x)=|y-x|$ by Lemma \ref{desilie}, then $\varphi\circ \ad_{y-x}=0$, thus the second term vanishes, and it also follows that
\begin{align*}
S(x,y)&=-\frac{|y-x|}{8}\max\limits_{\varphi\in N_{y-x}}\varphi([2x+y-x,[x,y]])= -\frac{|y-x|}{8}\max\limits_{\varphi\in N_{y-x}}\varphi([2x,[x,y-x]])\\
&=-\frac{|y-x|}{4}\max\limits_{\varphi\in N_{y-x}}\varphi([x,[x,y-x]])
\end{align*}
and with a similar manipulation we also obtain $S(x,y)=-\frac{|y-x|}{4}\max\limits_{\varphi\in N_{y-x}}\varphi([y,[y,y-x]])$.  Finally, if in \eqref{labe} we drop the term in $o(s)$ in and apply again Remark \ref{chainrule} we obtain the intermediate formula for $S(x,y)$.  
\end{proof}

First we state our curvature results in purely metric space terms: 

\begin{teo} Let $(G,\di)$ be a Lie group with a bi-invariant distance. Then
\begin{equation}\label{seck}
S(x,y)=6|y-x|\lim\limits_{r\to 0^+}\frac{\di(e^{r^2 x},e^{r^2 y})-r\di(e^{rx},e^{ry})}{r^4}.
\end{equation}
\end{teo}
\begin{proof}
Notation as in the previous proof, note first that for small $r>0$
\[
\frac{r\di(e^{rx},e^{ry})-\di(e^{r^2 x},e^{r^2 y})}{r^4}=\frac{|b(r^2)|-|b(r^4)|}{r^2}=\frac{|b(s)|-|b(s^2)|}{s}
\]
after changing variables $s=\sqrt{r}>0$. But
\[
||b(s^2)|-|b(0)||\le |b(s^2)-b(0)|=|s^2 b'(0)+ o(s)|
\]
hence
\[
\lim\limits_{s\to 0^+}\frac{|b(s)|-|b(s^2)|}{s}=\lim\limits_{s\to 0^+}\frac{|b(s)|-|b(0)|}{s}=\max\limits_{\varphi\in N_{y-x}}\varphi(b'(0))
\]
by Remark  \ref{chainrule}, and inspecting the proof of Proposition \ref{ese} we are done since this last term equals
\begin{equation}\label{otraS}
\frac{-1}{24}\max\limits_{\varphi\in N_{y-x}}\varphi([x,[x,y-x]])=\frac{S(x,y)}{6|y-x|}.\qedhere
\end{equation}
\end{proof}

\begin{rem}[$S$ in the Riemannian case]\label{rierem} Assume the group  $G$ has an $\Ad$-invariant Riemannian metric $\langle ,\rangle$ and $|\cdot|$ is the induced norm in $\g$. 
Let $0\neq x,y\in\g$ with $x\neq y$, then the only norming functional for $y-x$ is given by $\varphi(z) = \frac{1}{|y-x|}\langle z , y-x\rangle$ for $z\in \g$. Using Proposition \ref{ese} and the fact that $\ad x$ is skew-adjoint, it follows that 
\[S(x,y)=\frac{1}{4}|[x,y]|^2.\]
It's clear in this case that $S(x,y)=0$ if and only if $[x,y]=0$ (see Theorem \ref{fs}).

\end{rem}

Now we return to the general setting of a Lie group $G$ with a bi-invariant distance. We first show that the vanishing of $S$ implies its vanishing along the projection onto $\s$:

\begin{lem}\label{reduz} Let $x,y\in \g$. If $S(x,y)=0$, write  $x=x_c+x_k,y=y_c+y_k$ with $x_c,y_c\in Z(\g)$ and $x_k,y_k\in \s$ (Remark \ref{lasum}). Then $S(x_k,y_k)=0$.
\end{lem}
\begin{proof}
By the reversed triangle inequality and  the fact that $x_c,y_c$ are central, for $s>0$ we have
\begin{align*}
|y-x+\frac{1}{48}s[x+y,[x,y]]| &=|y_c-x_c+y_k-k_x+\frac{1}{48}s[x_k+y_k,[x_k,y_k]]|\\
& \ge |y_k-x_k+\frac{1}{48}s[x_k+y_k,[x_k,y_k]]|- |y_c-x_c|
\end{align*}
On the other hand $|y_k-x_k|\le |y-x|+ |y_c-x_c|$. Thus
\[
\frac{|y-x|-|(y-x)+\frac{1}{48}s[x+y,[x,y]]|}{s} \ge \frac{|y_k-x_k|-|y_k-x_k+\frac{1}{48}s[x_k+y_k,[x_k,y_k]]|}{s}\ge 0, 
\]
since the norm $|\cdot|$ restricted to $\s$ is also an $\Ad_K$-invariant norm there. By Proposition \ref{ese}, we see that it must be $0=\frac{S(x,y)}{|y-x|}\ge \frac{S(x_k,y_k)}{|y_k-x_k|}\ge 0$, thus $S(x_k,y_k)=0$ (note that if $y_k-x_k=0$, then there's nothing to prove).
\end{proof}

\begin{defi}Let $(G,\di)$ be a group with a bi-invariant distance. We say that the distance is strictly convex if there exists a $1$-neighbourghood $V$ such that $g,h\in V$ with $d(1,g)=d(1,h)$ and $\di(1,gh)=2\di(1,g)$ implies $g=h$.
\end{defi}

\begin{lem}\label{strcodi} Let $(G,\di)$ be a Lie group with a bi-invariant distance, let $|\cdot|$ be the induced Finsler norm. Then $\di$ is strictly convex if and only if $|\cdot|$ is strictly convex.
\end{lem}
\begin{proof}
Shrinking $V$ if necessary we have $g,h\in V$ implies $g=e^x,h=e^y$ with and $\di(1,g)=|x|$ and $\di(1,h)=|y|$. If $\di$ is strictly convex, take $x,y$ such that  $d=|x|=|y|$ and assume $|x+y|=|x|+|y|=2d$. In particular, there exists a unit norm functional $\varphi$ norming simultaneously $x$ and $y$. Renormalizing, we can assume that $g=e^x,h=e^y,e^xe^y\in V$ and moreover that $B(x,y)=x+y+z$ where $z$ is a series in brackets of $x,y$, hence each term of $z$ begins with $\ad x$ or with $\ad y$. In particular $\varphi(B(x,y))=\varphi(x+y+z)=|x|+|y|+0=|x+y|$. Then 
\[
|x+y|\ge \di(1,e^xe^y)=\di(1,e^{B(x,y)})=|B(x,y)|\ge \varphi(B(x,y))=|x+y|.
\]
This shows that $\di(1,gh)=|x+y|=2d=2\di(1,g)=2\di(1,h)$. The hypothesis tells us that $g=h$ or equivalently, that $x=y$, so the norm is strictly convex.

Now assume $|\cdot|$ is strictly convex. If $g,h\in V$,  let $d=|x|=|y|$. Then
\[
2d=\di(1,gh)=\di(1,e^xe^y)=\di(e^{-y},e^x)\le |x+y|\le |x|+|y|=2d,
\]
which implies that $x=y$, thus $g=h$, and this proves that the distance is strictly convex.
\end{proof}

\begin{teo}\label{strict}
If the distance in $G$ is strictly convex, then  $S(x,y)=0$ implies $[x,y]=0$.
\end{teo}
\begin{proof}
Writing $x=x_c+x_k$,$y=y_c+y_k$, by Lemma \ref{reduz} we see that $S(x_k,y_k)=0$. Let $\varphi=\langle z,\cdot\rangle$ be a functional in $\s'$ norming $y_k-x_k$ such that $\varphi([x_k,[x_k,y_k-x_k]])=0$ (Proposition \ref{ese}). Since the norm is strictly convex, it is also strictly convex restricted to $\s$, and by Remark \ref{strictc} it must be $F_{\varphi}=\{y_k-x_k\}$. By Theorem \ref{teoNbis} we see that $[x_k,y_k]=0$, but then we conclude that it must be $[x,y]=0$.
\end{proof}

So for strictly convex norms, $S(x,y)=0$ implies that the plane generated by $x,y$ is flat, i.e.
\[
\di(e^{s_1 x+t_1y},e^{s_2x+t_2y})=\di(1, e^{(s_2-s_1)x+(t_2-t_1)y})=|(s_2-s_1)x+(t_2-t_1)y|,
\]
as long as $(s_2-s_1)x+(t_2-t_1)y\in  B$, where $B$ is the ball in $\g$ such that $\exp|_B$ is a diffeomorphism onto its image (Corollary \ref{corodimenor}).

For non-strictly convex norms, the situation is much more interesting. In what follows we will discuss this, beginning with the following:

\begin{rem}\label{deriva}For given $x,y\in\g$, let $Lie(x,y)$ denote the closed Lie algebra generated by $x,y$, i.e., the smallest closed Lie subalgebra of $\g$ containing $x,y$. If $z$ commutes with $x,y$, then by means of the Jacobi identity we also obtain $[z,w]=0$ for any $w\in Lie(x,y)$ (thinking of $[z,\cdot]=\ad z$ as a derivation). 
\end{rem}

Recall  $F_{\varphi}$ is the exposed face given by any unit norm $\varphi$, and $C_{\varphi}$ is the cone generated by that exposed face (Definition \ref{caras}).

\begin{coro}\label{flat} If $x,y\in\g$ are sufficiently small and $x,y\in C_{\varphi}$, then
\[
\di(e^{x},e^{x+y})=|y|.
\]
\end{coro}
\begin{proof}
If $x+y,x$ are sufficiently small so the BCH series converges, in particular $x+y,x\in B$ (Corollary \ref{corodimenor}) we have $B(x+y,-x)=x+y-x+[x+y,f]+[x,g]=y+[y,f]+[x,\tilde{g}]$ for certain elements $f,g,\tilde{g}\in Lie(x,y)$ by Dynkin's formula. Then 
\begin{align*}
|y|& =|x+y-x|\ge \di(1,e^{x+y}e^{-x})=|B(x+y,-x)|\ge \varphi(B(x+y,-x))\\
& =|y|+ \langle z,[y,f]+ [x,\tilde{g}]\rangle=|y|+ \langle [z,y],f\rangle + \langle [z, x],\tilde{g}\rangle =|y|+0+0=|y|
\end{align*}
where we used \eqref{killcyc} and the previous remark (together with the fact that $\varphi$ norms $x,y$, therefore $z$ commutes with both $x,y$). 
\end{proof}


\begin{teo}[Flat sections]\label{fs} Let $x,y\in\g$. Consider the statements.
\begin{enumerate}
\item [(1)] $\varphi([x,[x,y]])=0$ for any $\varphi$ norming $y-x$.
\item[(2)] For sufficiently small $s> 0$,  the path $s^{-1}B(sy,-sx)$ is inside some exposed face $F_{\varphi}$ of the sphere of radius $|y-x|$ containing $y-x$.
\item[(3)] For sufficiently small $s\ge 0$, $\di(e^{sx},e^{sy})=s|y-x|$.
\item[(4)] $S(x,y)=0$.
\item[(5)] $\varphi([x,[x,y]])=0$ for some $\varphi$ norming $y-x$.
\end{enumerate}
Then $(1)\Leftrightarrow (2)\Leftrightarrow  (3) \Rightarrow (4)\Leftrightarrow (5)$. Moreover
\begin{enumerate}
\item[a)] If there exists only one norming functional of $y-x$ (the norm is smooth at $y-x$), then all the conditions are equivalent.
\item[b)] If in $(5)$ we have $F_{\varphi}=\{y-x\}$ (in particular, if the norm is strictly convex), then all the conditions are equivalent to $[x,y]=0$.
\end{enumerate}
\end{teo}
\begin{proof}
Assume $(1)$ and consider $\varphi=\langle z , \cdot \rangle$ adapted to $y-x$, which exists by Lemma \ref{adaptada}, then by Theorem \ref{teoNbis}, $[z,x]=0$. Since we also have $[z,y-x]=0$, we see that $[z,y]=0$. Hence $z$ commutes with every element of $Lie(x,y)$ (Remark \ref{deriva}), and arguing as in Corollary \ref{flat}, for small $s$ we conclude that
\[
s\varphi(y-x)=s|y-x| = |B(sy,-sx)| = \varphi(B(sy,-sx)).
\]
Dividing by $s$ it is clear that $(2)$ holds.

If $s^{-1}B(sy,-sx)$ is inside some exposed face of the sphere of radius $|y-x|$ containing $y-x$, we have $\varphi$ norming $y-x$ such that
\[
\varphi(s^{-1}B(sy,-sx))=|s^{-1}B(sy,-sx)|=|y-x|=\varphi(y-x),
\]
and since $\di(e^{sx},e^{sy})=|B(sy,-sx)|$, then $(3)$ holds. If $(3)$ holds, first we show that $(2)$ holds. To this end, consider $\gamma(s)=e^{sy}e^{-sx}=e^{B(sy,-sx)}$ for sufficiently small $s\in [0,s_0]$ where $(2)$ holds. Note that $\gamma$ joins $1, e^z$ with $z=B(s_0y,-s_0x)$, and its length is $s_0|y-x|$. By the hypothesis $(2)$ and Corollary  \ref{corodimenor}, if $s_0$ is small enough
\[
|z|=|B(s_0y,-s_0x)|=\di(e^{s_0x},e^{s_0y})=s_0|y-x|=L_0^{s_0}(\gamma).
\]
Then by Theorem 4.22(3) in \cite{lar19}, naming $\Gamma_s=B(sy,-sx)$, since $e^{\Gamma_s}=\gamma_s$, there exists $\varphi$ of unit norm such that 
\[
\varphi(B(sy,-sx))=|B(sy,-sx)| \quad \forall\, s\in [0,s_0].
\]
Then again by hypothesis (3) we have
\[
\varphi(B(sy,-sx))=|B(sy,-sx)|=\di(e^{sx},e^{sy})=s|y-x| 
\]
for $s\in [0,s_0]$, and this proves $(2)$. Note that also by  Theorem 4.22(3) in \cite{lar19} the last equality holds, in fact, for any norming functional $\varphi$ of $y-x$. If we compute the third lateral derivate with respect to $s\geq 0$ and then put $s=0$ it follows that $\varphi([x,[x,y]])=0$, which proves that in fact $(3)$ implies $(1)$. 

 Now assume that $(3)$ holds, it is clear from the very definition of $S$ that $S(x,y)=0$, so $(4)$ holds. From Proposition \ref{ese}  we see the equivalence of $(4)$ and $(5)$.

Now assume that the norm of $\g$ is smooth at $y-x$. Since $S(x,y)=0$, for the unique functional $\varphi=\langle z,\cdot\rangle$ norming $y-x$ we must have $\varphi([x,[x,y-x]])=\varphi([x,[x,y]])=0$. By Lemma \ref{adaptada}, $\varphi$ is adapted to $y-x$, and then all conditions are equivalent. Finally, if condition (5) holds for some norming $\varphi$ with $F_{\varphi}=\{y-x\}$, then $x,y-x$ commute because of Theorem \ref{teoNbis}, and then $[x,y]=0$. 
\end{proof}

\begin{rem} If $G=SU(2)$, and the eigenvalues of $y-x$ are equal then $y-x=\lambda 1$, hence  $[x,y]=[x,y-x]=0$ and all the conditions of the previous theorem are equivalent. Otherwise $y-x$ is regular and again all the conditions of the previous theorem are equivalent. This indicates a miscalculation in \cite[Example 4.4]{cocoeste}; that $\rho$ has $\|\rho\|_1=2.8\ne 1$.
\end{rem}

\begin{ejem}
Let $G=U(3)$, and denote $\langle x,y\rangle=\tr(xy^*)=-\tr(xy)$ which is (a constant multiple of) the opposite Killing form in $\mathfrak{su}(3)$. Consider $x,y,z\in \mathfrak{u}(3)$:
\[
x=\left(\begin{array}{ccc}
0 & 1 & 0  \\
-1 & 0 &  1  \\
0 & -1 & 0  
\end{array}  \right),
\quad y=\left(\begin{array}{ccc}
i & 1 & 0 \\
-1 & i & 1\\
0 & -1 & 0 \\
\end{array}  \right),
\quad v=y-x=i\left(\begin{array}{ccc}
1 & 0 & 0 \\
0 & 1 & 0 \\
0 & 0 & 0
\end{array}  \right),
\]
\[
z=i\left(\begin{array}{cccc}
1 & 0 & 0 \\
0 & 0 & 0  \\
0 & 0 &  0
\end{array}  \right), \qquad z_0=i\left(\begin{array}{cccc}
1/2 & 0 & 0 \\
0 & 1/2 & 0  \\
0 & 0 &  0
\end{array}  \right),\qquad  \varphi=\langle z,\cdot \rangle,\quad \varphi_0=\langle z_0,\cdot \rangle \in \mathfrak{u}(3)'.
\]
We put in $\mathfrak{u}(3)$ the $\Ad$-invariant norm $|v|=\|v\|_{\infty}=\max\{|\lambda_i(v)|\}$ (the spectral norm). It is well known that its dual norm is the trace norm $\|z\|_1=\tr|z|=\sum_i |\lambda_i(z)|$. Then we have $\|v\|_{\infty}=1$, $\|\varphi\|=\|z\|_1=1=\|\varphi_0\|=\|z_0\|_1$ and
\[
\varphi(v)=\langle z,v\rangle=1=\|v\|_{\infty}=|v|=\varphi_0(v),
\]
therefore $\varphi$ and $\varphi_0$ norm $y-x$. It is also clear that $\varphi_0$ is adapted to $y-x$. Now $P_vx$ is the co-diagonal part of $x$, described as
\[
P_v x=\left(\begin{array}{ccc}
0 & 0 & 0  \\
0 & 0 &  1  \\
0 & -1 & 0  
\end{array}  \right),
\]
since the other part of $x$ commutes with $v$. A straightforward computation shows that $[P_vx,z]=0$ and that $[P_vx,z_0]\ne 0$, therefore by Theorem \ref{teoNbis} we have 
\[
\varphi([x,[x,v]])=\varphi([x,[x,y-x]])=0\quad \textrm{ while }\quad \varphi_0([x,[x,y-x]])\ne 0.
\]

This proves that in general the conditions $(1)$, $(2)$ and $(3)$ from the last theorem are not equivalent to the conditions $(4)$ and $(5)$.

\end{ejem}

\subsubsection{Sectional curvature}\label{sec:curvat}

We now discuss in more detail the possibility of defining a notion of metric curvature for a $2$-plane in $\g$.

\begin{rem} In the Riemannian setting, if we consider a 2-plane $\pi=span\{x,y\}\subseteq \g$ the sectional curvature is given by
\[\sec(\pi)=\frac{\langle R(x,y)y,x\rangle}{A^2(x,y)}=\frac{S(x,y)}{A^2(x,y)},\]   
where $R$ is the curvature tensor and $A^2(x,y)= \langle x,x \rangle \langle y,y \rangle - \langle x , y \rangle ^2$ is the squared area of the parallelogram generated by $x$ and $y$. Since the sectional curvature does not depend on the basis choosen, it's easy to see that 
\[\sec(\pi)=\max\limits_{x,y\in\pi, |x|=|y|=1} S(x,x+y),\]
which provides an expression that eliminates the dependence on the area.
\end{rem}

\begin{defi}[Sectional curvature of a $2$-plane $\pi\subset \g$] Let $|\cdot|$ be any Finsler $\Ad$-invariant norm in $\g$, let $\pi$ be $2$-plane in $\g$. We define
\[
\sec(\pi)=\max\limits_{x,y\in\pi, |x|=|y|=1} S(x,x+y)=\nicefrac{1}{4}\max\limits_{|x|=|y|=1}\min\limits_{\varphi\in N_y}\varphi([x,[y,x]])
\]
which is non-negative by Proposition \ref{ese}.
\end{defi}

\begin{teo}\label{flats} For any $2$-plane $\pi\subset \g$ we have
\begin{enumerate}
\item $\sec(\pi)=0$ if and only if $S(x,y)=0$ for any $x,y\in\pi$.
\item With the normalization $|[x,y]|\le  2|x|\,|y|$ we  have $0\le  \sec(\pi)\le 1$.
\end{enumerate}  
\end{teo} 
\begin{proof}
If $\sec(\pi)=0$, since $\varphi([x,[y,x]])\ge 0$ when $\varphi\in N_y$ (for any unit norm $x,y\in \pi$), then it must be that $\min\limits_{\varphi\in N_y}\varphi([x,[y,x]])=0$ for all $|x|=|y|=1$. We claim that for  $x,y\ne 0$ in $\pi$, there exists $\varphi$ norming $y$ such that $\varphi([x,[x,y]])=0$: indeed, let  $y_0=\frac{y}{|y|}$ and likewise $x_0=\frac{x}{|x|}$,  then there exists $\varphi$ norming $y_0$ such that $\varphi([x_0,[x_0,y_0]])=0$. But this $\varphi$ also norms $y=|y|y_0$, and by bi-linearity in $x$ we have the claim, and then $S(x,y)=0$ for any pair of $x,y\in \pi$. 

If $S(x,y)=0$ for all $x,y\in \pi$ then clearly $\sec(\pi)=0$. Conversely, let $U=x$, $V=x+y$, we have that $\varphi$ norms $V-U$ and $\varphi([U,[U,V-U]])=0$, thus by Theorem \ref{ese}, we have $S(U,V)=0$. Since these are arbitrary vectors in $\pi$ the first claim of the theorem is proved. To prove the second assertion, note that with the chosen normalization we have 
\begin{align*}
\sec(\pi)&=\frac{1}{4}\max\limits_{|x|=|y|=1}\min\limits_{\varphi\in N_y}\varphi([x,[y,x]])\leq \frac{1}{4}\max\limits_{|x|=|y|=1}|[x,[y,x]]|\\
&\leq\frac{1}{4}\max\limits_{|x|=|y|=1}4|x||y||x|=1.\qedhere
\end{align*}
\end{proof}

\begin{coro}\label{secvsnorm}
If $\sec(\pi)=0$ and the unit ball of $\g$ has a strictly convex point that lies in $\pi$, then $\pi$ is abelian.
\end{coro}
\begin{proof}
If $V$ is a point of strict convexity of the unit ball of $\g$, take $U\in\pi$ linearly independent and let $x=U$, $y=U+V$. Then  $S(x,y)=0$ and by Theorem \ref{fs} we have $[x,y]=0$, since $y-x=V$. Thus $\pi$ is abelian.
\end{proof}

\begin{rem}[The Riemannian case, revisited] In the Riemannian setting if $\dim(\g)=d$ and the normalization of the previous theorem  holds, we have that $0\le\sec\le 1$ for any $2$-plane. We can then compare the lengths of the sides of the geodesic triangle $\Delta(a,b,c)\subset G$ and the corresponding geodesic triangle in the unit sphere of dimension $d$, that is  
$$
\Delta'(a',b',c')\subset S=\{V:|V|=1\}\subset H
$$
where $\dim(H)=\dim (\g)+1=d+1$. Using the Cartan-Alexandrov-Toponogov Comparison Theorem \cite[Theorem 6.5.6]{bbi}, by fixing the lengths of two sides of the triangle in $G$ and the angle substended between them, it follows that  
$$
\di(a,b)\le \di_S(a',b'),
$$
as long as sum of the sides of the triangle is shorter than $2\pi$. On the left we have the Riemannian distance induced by the $\Ad$-invariant norm on $G$; on the right, the Riemannian distance on $S$ induced by ambient metric of $\mathbb R^{d+1}$. Using another comparision triangle (now fixing the lengths of all three sides but not the angle) and the previous inequality it follows that \textit{the sum of the inner angles of a geodesic triangle in $G$ with the Riemannian metric is less than $3\pi$}. On the other hand, since $\sec\geq 0$ for any $2$-plane we can use the same arguments but now comparing the space $G$ with a flat space of dimension $d$ obtaining that: \textit{the sum of the inner angles of a geodesic triangle in $G$ with the  Riemannian metric is greater or equal than $\pi$}.
\end{rem}

\section*{Acknowledgements} This research was supported by CONICET, ANPCyT and Universidad de Buenos Aires, grant UBACyT 20020220400256BA.

\end{document}